\documentclass[11pt,psamsfonts]{article}
\usepackage{a4,fullpage,amssymb,epsf,psfrag,times}
\usepackage{graphicx}
\usepackage{amsmath}
\usepackage{amsfonts}
\usepackage{enumerate}
\usepackage{latexsym}
\usepackage{epsfig}
\usepackage{amsthm}
\usepackage{amsmath}
\usepackage{amssymb}
\usepackage{latexsym}
\usepackage{epsfig}
\usepackage{amssymb,latexsym}
\usepackage[all]{xy}
\usepackage{graphicx}
\usepackage{amsmath}
\usepackage{amsfonts}
\usepackage{enumerate}
\usepackage{latexsym}
\usepackage{epsfig}
\usepackage{amsthm}
\usepackage{amsmath}
\usepackage{amssymb}
\usepackage{latexsym}
\usepackage{epsfig}

\makeatletter\@addtoreset{figure}{section}\makeatother
\makeatletter\@addtoreset{table}{section}\makeatother

\newtheorem{theorem}{Theorem}[]

\newtheorem{lemma}[theorem]{Lemma}

\newtheorem{cor}[theorem]{Corollary}

\newcommand{\id}{\mathop{{\rm id}}\nolimits}
\newcommand{\R}{{\mathbb R}}
\newcommand{\C}{{\mathbb C}}
\newcommand{\Z}{{\mathbb Z}}
\newcommand{\Q}{{\mathbb Q}}
\newcommand{\N}{{\mathbb N}}

\newcommand{\op}[1]{\!\!\mathop{\rm ~#1}\nolimits}

\newcommand{\scriptop}[1]{\!\!\mathop{\mbox{\rm \scriptsize ~#1}}\nolimits}

\newenvironment{remark}{\refstepcounter{theorem}\par\medskip\noindent{\bf
Remark~\thetheorem~~}}{\unskip\nobreak\hfill\hbox{ $\oslash$}\par\bigskip}

\begin{document}

\title{Fixed points of  symplectic periodic flows}

\author{Alvaro Pelayo \footnote{partially supported by an NSF postdoctoral fellowship;
this work started when he was affiliated with MIT.}\,\,  
and Susan Tolman\footnote{partially supported by National Science Foundation
Grant DMS \#07-07122.}}

\date{}

\maketitle

\begin{abstract}
The study of fixed points is a classical subject in geometry and dynamics.
If the circle acts in a Hamiltonian fashion on a compact symplectic manifold $M$, then
it is classically known that there are at least $\frac{\op{dim}M}{2}+1$ fixed
points; this follows from Morse theory for the momentum map of the action.
In this paper we use
Atiyah\--Bott\--Berline\--Vergne (ABBV) localization in equivariant cohomology
to prove that this conclusion also holds for 
symplectic circle actions with  non\--empty fixed sets,
as long as the Chern class map is somewhere injective
--- the Chern class map assigns to a fixed point the sum
of the action weights at the point. 
We complement this result with
less sharp lower bounds on the number of fixed
points, under no assumptions;  
from a dynamical systems viewpoint,  our results imply
that there is no symplectic periodic flow with exactly one or two equilibrium
points on a compact manifold of
dimension at least eight.
\end{abstract}

\section{Introduction}

The study of fixed points of flows or maps is a classical 
and important topic studied in geometry and dynamical systems. 
For instance, in the context of symplectic geometry, 
it follows from the works of 
Atiyah \cite{atiyah}, Guillemin\--Sternberg \cite{gs}, 
and Kirwan \cite{kirwan}
that a lot of the information 
of a symplectic manifold equipped
with a Hamiltonian torus action is
encoded by the fixed point set of the  action. For example, if
the torus has precisely half the dimension
of the manifold, then all the information
is encoded by the fixed point  set; see
Delzant \cite{delzant}.

In this paper we study 
symplectic (but not necessarily Hamiltonian) circle actions
on compact symplectic manifolds with non-empty fixed point sets.
Specifically, we are interested in providing lower bounds
on the number of fixed points of the circle action. Equivalently, out goal is
to give lower bounds for the number of equilibrium
points of symplectic periodic flows. 

There are many non\--Hamiltonian circle actions with no fixed points; for
instance,  given an even\--dimensional torus, consider the natural action of any
circle subgroup. 
A well\--known theorem of McDuff  \cite[Prop.~2]{McDuff} says 
that a symplectic circle action on a four\--dimensional
compact symplectic manifold with at least one fixed point
is Hamiltonian.  
Non\--Hamiltonian symplectic circle actions on higher dimensional
manifolds  are not well
understood. 
A manifestation of this fact is that
it is not known whether there is a non\--Hamiltonian symplectic circle action
on  a compact symplectic manifold $M$
so that the fixed set $M^{S^1}$ is discrete but not empty.
(In \cite[Prop.~1 and Sec.~3]{McDuff},  McDuff also constructs  
a non\--Hamiltonian symplectic circle action on a six\--dimensional compact
symplectic manifold, but the fixed sets are tori.)

It is a classical fact
that  a Hamiltonian circle action on a compact 
symplectic manifold $(M,\, \omega)$ has at least 
$\frac{1}{2} \dim M + 1$
fixed points.
This is an easy consequence of the fact that, if the fixed set is discrete,
then the  momentum map $\mu \colon M \to \R$  is a perfect Morse function
whose critical set is the fixed set.
Therefore,  the number of fixed points is equal to the rank of $\sum_i \op{H}_i(M;\R)$. Finally,
this sum is at least
$\frac{1}{2}\dim M +1$ 
because 
$$ [1],\, [\omega],\, 
\big[\omega^2 \big], \ldots, \big[ \omega^{\frac{1}{2} \dim M}\big]$$
are distinct cohomology
classes. 
In addition, a lot of  properties of Hamiltonian
circle actions are known, see McDuff\--Tolman \cite{mcdufftolman} and the
references therein.
A natural question arises: what can one say about the fixed point 
set of a 
{\em symplectic} circle action?
The goal of this paper is to shed light into this question. Before stating
our main results
we need to recall some terminology.

\vspace{1mm}
A \emph{symplectic form} on a compact manifold $M$ is a closed, non\--degenerate
two\--form $\omega \in \Omega^2(M)$.
A circle action on $M$ is \emph{symplectic} if it preserves $\omega$.
A symplectic circle action is \emph{Hamiltonian} if there exists 
 a map $\mu \colon M \to \R$ such that
$$- \op{d}\! \mu = \iota_{X_M} \omega,$$
where $X_M$ is the vector field on $M$ induced by the circle action.
The map $\mu$ is called the \emph{momentum map}.
Since $\iota_{X_M} \omega$ is closed, every symplectic action is
Hamiltonian if $\op{H}^1(M;\, \R) = 0$.

Let
the circle  act on a compact symplectic manifold $(M,\, \omega)$
with momentum map $\mu \colon M \to \R$.
Since the set of compatible almost complex structures
$J \colon \op{T}\!M \to \op{T}\!M$ is contractible,  there is a well\--defined
multiset of integers, called \emph{weights}, associated to  each fixed point $p$.
Moreover,
let
 $\op{c}_1(M)|_p$ denote the first equivariant Chern class of the tangent bundle of $M$
 at  $p \in M^{S^1}$, which we can naturally identify with an integer 
 $\op{c}_1(M)(p)$: the sum
 of the weights at $p$; (see Section \ref{sec2} for  
a brief explanation.)
Under this identification,
 we consider the map
 $$
 \op{c}_1(M) \colon M^{S^1} \to \Z,\,\,\,\,\,\,\,\,\, p \mapsto \op{c}_1(M)(p) \in \Z.
 $$
We refer to the map $\op{c}_1(M)$ as
 the \emph{Chern class map} of $M$.
Finally,
let $X$ and $Y$ be sets and let $f\colon X \to Y$ be a map. 
We recall that
$f$ is \emph{somewhere injective} 
if 
there is a point $y\in Y$ such that $f^{-1}(\{y\})$
is the singleton. 
Note that if $f$ is a somewhere injective map, then necessarily
$X \neq \emptyset$. 
(In particular, if the Chern class map is somewhere
injective, $M^{S^1} \neq \emptyset$.)
On the other hand, if $X \neq \emptyset$,
then  every injective map is somewhere injective.

\begin{theorem}\label{general}
Let the circle  act symplectically
on a compact  symplectic manifold $M$.
If  the Chern class map is somewhere injective,
then the circle action has at least 
$\frac{1}{2}\dim M + 1$
fixed points.
\end{theorem}

There has been a considerable amount of interest,
as far as we know with no success, in trying to construct non\--Hamiltonian symplectic
circle actions with non\--empty finite fixed point set. The theorem above
tells us where we should \emph{not look} for examples.
An interesting open question is whether
Theorem \ref{general} holds without the assumption on the Chern class map;
we suspect that the answer is ``yes''.
Our proof uses the Atiyah\--Bott and Berline\--Vergne localization formula in equivariant
cohomology, following the ideas of the paper by Tolman\--Weitsman \cite{tolmanweitsman}
on semifree actions. We also use some
basic properties of Poincar\'e polynomials and Morse functions to prove:

\begin{theorem} \label{2}
If the circle acts symplectically on  a compact symplectic manifold $M$
with isolated fixed 
points, then
$$\sum_{p \in M^{S^1}} \op{c}_1(M)(p) = 0.$$
\end{theorem}

In their paper, Tolman and Weitsman pointed out 
that their results did not rule out the existence of an example of a symplectic
circle action on a $6$\--dimensional manifold \emph{with exactly two
fixed points with weights $(1,\,1,\,-2)$ and $(-1,\,-1,\,2)$} and the question
has been open ever since. Note that our results above do not rule out the existence of such an example either.
In this paper we shall prove:

\begin{theorem}\label{twopoints2}
Let the circle act symplectically on a compact, 
connected
symplectic manifold $M$
with exactly two fixed points. 
Then either $M$ is the $2$-sphere or $\dim M  = 6$ and 
there exist natural numbers $a$ and $b$ so that the weights at the two fixed points
are $\{a,b,-a-b\}$ and $\{a+b, -a,-b\}$.
\end{theorem}

As a conclusion from these three theorems we obtain:

\begin{cor}\label{maintheorem}
Let the circle act symplectically
on a compact  symplectic manifold $M$ with
non\--empty fixed point set. 
Then
there are at least  two fixed points, and if 
$\dim M\ge 8$, 
then  there are at least three fixed points.
Moreover, if 
the Chern class map is not identically zero and 
$\dim M\ge 6$, 
then there are at least four fixed points.
\end{cor}

\begin{proof}
By Theorem \ref{general} there cannot be exactly one fixed point; 
since
$M^{S^1} \neq \emptyset$ this implies that there are least two fixed points.
If $\dim M \ge 8$ then Theorem \ref{twopoints2} implies that there are at least three fixed points.
If the Chern class map is not identically zero, it
follows from Theorem \ref{2} that it is not constant, 
that is,  
$$ \Big\vert \{\op{c}_1(M)(p) \mid p \in M^{S^1}\} \Big\vert \geq 2.
$$
Hence, if  there are less than 
$\frac{1}{2}\dim M  + 1$
fixed points, then
Theorem \ref{general} implies that there must be at least four.
\end{proof}

\begin{remark}
We do not know if there is a symplectic circle
action on a compact, connected symplectic manifold $M$ with
exactly {\em three} fixed points, other than the standard
actions on $\C P^2$.  Note that, by Corollary~\ref{odd},
if this does occur then $\frac{1}{2}\dim M$ is even.
\end{remark}

There are several works on (possibly
non\--Hamiltonian) symplectic circle
actions that are related to this note, including Godinho \cite{go} and Weitsman \cite{we}.
For works on (possibly non\--Hamiltonian) symplectic $(S^1)^k$\--actions, $k>1$,
see Duistermaat\--Pelayo \cite{dp} and Pelayo \cite{pe}; in these cases
the fixed point sets are well understood.
In addition,
the study of $S^1$\--actions and their  flows 
is central in the analysis of bifurcation diagrams
in the modern global theory of
integrable systems, see \cite{vungoc}, \cite{pelayovungoc1}, \cite{pelayovungoc2}.
A \emph{symplectic periodic flow} is a flow generated by
a symplectic circle action. Fixed points of the circle
action correspond to equilibrium points of the
flow. From a dynamical systems viewpoint,
Corollary \ref{maintheorem} implies:

\begin{cor} \label{1:thm}
There is no symplectic periodic flow with exactly one equilibrium
point on a compact manifold $M$.
Moreover if $\dim M\ge 8$, then there is no symplectic periodic flow 
with exactly  two equilibrium points.
\end{cor}

\emph{Acknowledgements.} Special thanks to D. McDuff
for comments and giving us a stronger formulation
of Theorem \ref{general}.
Thanks to D. Auroux, V. Guillemin
and S. Sabatini for comments on a preliminary version. 
Finally, we are grateful to an anonymous referee for 
interesting comments and suggestions.

\section{Proofs of Theorems \ref{general} and \ref{2}} \label{sec2}

Before beginning these proofs,  we briefly introduce some background material.
Let the circle act on a manifold $M$.
The \emph{equivariant cohomology of $M$}
is, by definition,
$\op{H}^*_{S^1}(M) =\op{H}^*(M \times_{S^1} S^\infty).$
For example, if $p$ is a point then
$\op{H}^*_{S^1}(p;\, \Z) = \op{H}^*(\mathbb{C}P^\infty;\, \Z) = \Z[t].$
If $V$ is an equivariant vector bundle over $M$, 
then the \emph{equivariant Euler  class of } $V$ is the Euler class
of the vector bundle $V \times_{S^1} S^\infty$ over $M \times_{S^1} S^\infty$.
The \emph{equivariant Chern classes} of equivariant complex vector bundles
are defined analogously.

If $M$ is oriented and compact then 
the projection map $\pi \colon M \times_{S^1} S^\infty \to \mathbb{C}P^\infty$
induces a natural push\--forward map 
$$\pi_* \colon \op{H}_{S^1}^{i}(M;\Z)
\to \op{H}^{i - \dim M} (\mathbb{C}P^\infty;\, \Z).$$ 
Note, in particular
that if $i < \dim M$ then $\pi_*(\alpha) = 0$ for all $\alpha \in \op{H}^i_{S^1}(M;\Z)$.
Since this map is given by ``integration over
the fiber,'' we will usually denote it by the symbol $\int_M$.
We will need the following theorem, due to Atiyah\--Bott and
Berline\--Vergne \cite{AB,BV}.

\begin{theorem}[ABBV Localization] \label{ABBV}
Let the circle act on a compact oriented manifold $M$.
Fix $\alpha \in \op{H}_{S^1}^*(M;\, \Q)$. As elements of $\Q(t)$,
$$\int_M \alpha = \sum_{F \subset M^{S^1}} \int_F \frac{ \alpha|_F}{\op{e}_{S^1}(\op{N}_F)},$$
where the sum is over all fixed components,
and $\op{e}_{S^1}(\op{N}_F)$ denotes the equivariant Euler class of the normal
bundle to $F$.
\end{theorem}

Now let the circle act symplectically on a symplectic manifold  $(M,\, \omega)$, and 
let $J \colon \op{T}\!M \to \op{T}\!M$ be a compatible almost
complex structure.
If $p \in M^{S^1}$ is an isolated fixed point, then there are
well-defined non-zero (integer) weights 
$\xi_1,\,\ldots,\,\xi_n$ in the isotropy representation $\op{T}_p M$ (repeated with multiplicity). 
Indeed, there exists an identification of  $\op{T}_pM$ with $\mathbb{C}^n$,
where the $S^1$ action on $\mathbb{C}^n$ is
given by  $\lambda \cdot (z_1,\ldots,z_n)=
(\lambda^{\xi_1}z_1,\ldots,\lambda^{\xi_n}z_n)$; the
integers $\xi_1,\ldots, \xi_n$ are determined, up to permutation,
by the $S^1$\--action and the symplectic form.
Therefore, restriction of
the  $i^{\scriptop{th}}$\--equivariant Chern class to a fixed point $p$ is given by
$$
\op{c}_i(M)|_p = \sigma_i(\xi_1,\ldots,\xi_n) \, t^i
$$
where $\sigma_i$ is the $i$'th elementary symmetric polynomial
and $t$ is the generator of $\op{H}^2_{S^1}(p;\Z)$.
For example, $\op{c}_1(M)|_p = \sum \xi_i t$
and  the equivariant Euler class 
of the tangent bundle at $p$
is given by
$
\op{e}_{S^1}(\op{N}_p) =\op{c}_n(M)|_p =  \left( \prod \xi_j \right) \, t^n.
$
Hence,
$$\int_p \frac{\op{c}_i(M)|_p}{\op{e}_{S^1}(\op{N}_p)}= \frac{ \sigma_i(\xi_1,\ldots,\xi_n) }
{  \prod \xi_j } \, t^{i -n}.$$

\begin{lemma}\label{Azero}
Let the circle  act symplectically
on a compact symplectic $2n$\--manifold $M$ with isolated fixed points.
If the range of the Chern class map $\op{c}_1(M)$ contains at most $n$ elements, 
then
$$
\sum_{\substack{p \in M^{S^1}\\ 
\op{c}_1(M)(p)=k}} \! \!  \frac{1}{\,\Lambda_p} = 0 \qquad
\forall \ k \in \Z.
$$
Here, $\Lambda_p$ is the product of the weights (with multiplicity) in the isotropy representation 
$\op{T}_p M$ for all $p \in M^{S^1}$.
\end{lemma}

\begin{proof}
Let 
\begin{gather*}
\left\{\op{c}_1(M)(p) \; \middle\vert \; p \in M^{S^1} \right\}=
\left\{k_1,\, \ldots,\, k_{\ell}\right\} \subset \Z,
\quad \mbox{and define} \\
A_i:=\sum_{\substack{p \in M^{S^1}\\ 
\op{c}_1(M)(p)=k_i}} \! \!  \frac{1}{\,\Lambda_p}
\quad \forall \ i \in \{1,\dots,\ell\}.
\end{gather*}
Consider the $\ell \times \ell$
matrix $\mathcal{B}$  given by
$$
\mathcal{B}_{ij}:= (k_i)^{j-1} \qquad \forall \ 1 \leq i \leq \ell \ \mbox{and} \ 1 \leq j \leq \ell.
$$
Since  $\ell \leq n$ by assumption, 
$$
\int_M \op{c}_1(M)^j = 0 \,\,\, \textup{for all}\,\,\, j < \ell.
$$
Therefore, applying the ABBV localization formula (Theorem \ref{ABBV})
to
$$1, \,\op{c}_1(M), \dots, \,\op{c}_1(M)^{\ell - 1}$$
gives
a homogenous system of linear equations
\begin{equation} \label{linear}
\mathcal{B}\cdot (A_1,\,\ldots,\,A_{\ell})=(0,\,\ldots,\,0).
\end{equation}
Since $\mathcal{B}$ is a Vandermonde matrix, we have that 
$$\op{det}(\mathcal{B}(\ell)) \neq 0.$$ 
Thus, it follows from (\ref{linear}) that 
$$
A_1= \dots =A_{\ell}\,=\,0.
$$
\end{proof}

\begin{cor}\label{notone}
There does not exist a  symplectic $S^1$\--action with exactly
one fixed point.
\end{cor}

\begin{proof}[Proof of Theorem \ref{general}]
If the Chern class map is somewhere injective, then
there exists $k \in \Z$ so that
$$\sum_{\substack{p \in M^{S^1}\\ 
\op{c}_1(M)(p)=k}} \! \!  \frac{1}{\,\Lambda_p} \neq 0.$$
By Lemma \ref{Azero}, this
implies that the range of the Chern class map contains at least
$\frac{1}{2} \dim M + 1$ elements;
a fortiori,
the action has at least $\frac{1}{2} \dim M + 1$ fixed points.
\end{proof}

\begin{remark}
More generally, let $s \colon M^{S^1} \to \Z$ be given by any product of Chern classes
of total degree $2m$.
Then, by a nearly identical argument, if the range of $s$ contains at
most $\frac{n}{m}$ elements  then
$$
\sum_{\substack{p \in M^{S^1}\\ 
s(p) =k}} \! \!  \frac{1}{\,\Lambda_p} = 0 \qquad
\forall \ k \in \Z.
$$
In particular, $s$ is nowhere injective.
\end{remark}

\begin{lemma} \label{1}
Let the circle act symplectically
on a compact symplectic $2n$\--manifold with isolated fixed points.
Then
\begin{eqnarray*} 
\Big\lvert \big\{ p \in M^{S^1} \ \big| \ \lambda_p = i \big\} \Big\rvert = 
\Big\lvert \big\{ p \in M^{S^1} \ \big| \ 
 \lambda_p = n - i \big\} \Big\rvert \quad 
\forall \ i \in \Z.
\end{eqnarray*}
Here,  $\lambda_p$ is the number of negative weights at  $p$
for all $p \in M^{S^1}$.
\end{lemma}

\begin{proof}
By perturbing and rescaling 
the symplectic form if necessary, we may assume that there exists a
circle valued moment map $\Phi \colon M \to S^1$ 
and that zero is a regular value \cite[Lemma 1]{McDuff}.

\vspace{1mm}

By cutting this manifold at $0$, we can construct a new symplectic orbifold
$N$ with momentum map 
$
\Psi \colon N \to \R
$
with the property that the minimum and maximum
fixed sets are each diffeomorphic  to the reduced space 
$$
M_0 = \mu^{-1}(0) /S^1.
$$ 
On the complement of
these extremal fixed point sets, $N$ is equivariantly symplectomorphic to $M \smallsetminus \mu^{-1}(0)$.

\vspace{0.1in}

Since $\Psi$ and $-\Psi$ are  perfect Bott\--Morse functions, we have
equalities of Poincar\'e polynomials
\begin{align}
\label{poincpoly}
\op{P}_{N} & = \op{P}_{M_0} + \sum_{p \in M^{S^1} } t^{\lambda_p} + t \, \op{P}_{M_0} \qquad \mbox{and} \\
\label{poincpoly2}
\op{P}_N  &= \op{P}_{M_0} + \sum_{p \in M^{S^1} } t^{n - \lambda_p} + t \, \op{P}_{M_0}.
\end{align}
The claim follows from 
(\ref{poincpoly}) and (\ref{poincpoly2}).
\end{proof}

By Poincar\'e duality, this has the following corollary
\begin{cor}\label{odd}
Let the circle act symplectically on a compact symplectic $2n$-manifold
with $k$ isolated fixed points.  If $k$ is odd, then $n$ is even.
\end{cor}

We are now ready to prove our final result.

\begin{lemma} \label{last:lem}
Let the circle  act symplectically
on a compact symplectic manifold $M$ with isolated fixed points.
Then
$$
\sum_{p \in M^{S^1}} \op{N}_p(\ell) = \sum_{p \in M^{S^1}} \op{N}_p(-\ell) \quad \quad 
\forall \ \ell \in \Z.
$$
Here,  $\op{N}_p(\ell)$ is the
multiplicity of $\ell$ in the  isotropy representation $\op{T}_p M$ for all weights $\ell \in \Z$
and all $p \in M^{S^1}$.
\end{lemma}

\begin{proof}
Assume that the lemma is true for all manifolds of dimension less than $\dim M$;
we will prove that it holds for $M$. 

\vspace{1mm}

Without loss of generality we may assume that $M$ is connected.
By quotienting out by the subgroup which acts trivially,
we may assume that the action is effective.
Given a natural number  $\ell \neq 1$,
let $\Z/(\ell) \subset S^1 \subset \C$ be the natural subgroup generated by
$e^{\frac{2 \pi i}{\ell}}$, and
let $M^{\Z/(\ell)} \subset M$ denote the set of points which are fixed by $\Z/(\ell)$.
Given  any component $Z \subset M^{\Z/(\ell)}$
and any $p \in Z^{S^1}$, the multiplicity of $\ell$ (or $-\ell$)
in the isotropy representation
$\op{T}_p Z$
is equal to the multiplicity of $\ell$ (or $-\ell$) in $\op{T}_p M$.
Since the action is effective, $\dim Z < \dim M$.
Therefore, by the inductive hypothesis,
\begin{equation*} 
\sum_{p \in Z^{S^1}} \, \op{N}_p(\ell) = \sum_{p \in Z^{S^1}} \, \op{N}_p(-\ell) \quad \quad \forall \ \ell \in \Z \smallsetminus \{-1,0,1\}.
\end{equation*}
Since every fixed point lies in a unique component of $M^{\Z/(\ell)}$,
it follows immediately that
\begin{gather*}
\sum_{p \in M^{S^1}} \, \op{N}_p(\ell) = \sum_{p \in M^{S^1}} \, \op{N}_p(-\ell)  \quad \quad
\forall \ \ell \in \Z \smallsetminus \{-1,0,1\}.
\end{gather*}
Moreover, 
Lemma \ref{1} implies
that
$$
\sum_{p \in M^{S^1}} \sum_{\ell =1}^\infty \, \op{N}_p(-\ell) = \sum_{p \in M^{S^1}} \lambda_p
= \sum_{p \in M^{S^1}} \Big( \textstyle{\frac{1}{2}}\displaystyle \dim M - \lambda_p \Big) 
= \sum_{p \in M^{S^1}} \sum_{\ell =1}^\infty \, \op{N}_p(\ell).
$$
Here,  $\lambda_p$ is the number of negative weights at  $p$
for all $p \in M^{S^1}$.
The result follows immediately.
\end{proof}

\begin{proof}[Proof of Theorem \ref{2}] 

As we explained in the beginning of this section,
$$\sum_{p \in M^{S^1}} \op{c}_1(M)(p) =
\sum_{\ell \in \Z} \sum_{p \in M^{S^1}} \ell \, \op{N}_p(\ell)
.$$
But Lemma \ref{last:lem} implies that the right hand side of this equation is 
equal to zero.
\end{proof}

\section{Proof of Theorem \ref{twopoints2}}

In the proof, we will show that  the set of weights in
the isotropy representation $\op{T}_p M$  contains a natural subset which
satisfies the assumptions of the following technical lemma.
It is easy to check that every such set contains at least
three non-zero integers. With a little more work, we can
describe these sets precisely.

\begin{lemma} \label{multi}
Let $\Sigma$ be a multiset of non\--zero integers. Let 
$\op{N}(\ell)$ denote the multiplicity of $\ell$ in $\Sigma$
for all $\ell \in \Z$.
Assume that the following hold.
\begin{itemize}
\item $\op{N}(\ell) \op{N}(-\ell) = 0$ for all $\ell \in \Z$.
\item $\Sigma = -\Sigma \mod \alpha$  for all $\alpha \in \Sigma.$
\item $\sum_{\alpha  \in \Sigma} \alpha = 0$.
\item There exists $n \in \N$ such that $\op{N}(n) = 1$
and $\op{N}(\ell) = 0$ for all $\ell \in \Z$ such that $|\ell| > n$.
\end{itemize}
Then there exist
natural numbers $a$ and $b$ so that $$\Sigma = \{a+b,-a,-b\}.$$
\end{lemma}

\begin{proof}
Let 
$k$ be the second largest number such that $\op{N}(k) + \op{N}(-k) \neq 0$.
By assumption, there exist bijections
$\theta \colon \Sigma \to \Sigma$ and $\theta' \colon \Sigma \to \Sigma$
such that $\theta(\alpha)  + \alpha =  0 \mod n$
and $\theta'(\alpha) + \alpha  = 0 \mod k$ for all $\alpha \in \Sigma$.
Furthermore, we may assume that 
$$\theta^2 = (\theta')^2 = \id_{\Sigma}.
$$

Define multisets
$$\Sigma^+ = \{\alpha \in \Sigma  \mid \alpha > 0\} \smallsetminus \{n\}
\quad \mbox{and} \quad
\Sigma^- = \{\alpha \in \Sigma \mid \alpha < 0\}.$$

Let $\alpha \in \Sigma^+$.
By assumption, $-\alpha \not\in \Sigma$;
in particular, $\theta(\alpha) \neq - \alpha$.
Therefore, since $0 < \alpha < n$ and  $-n  < \theta(\alpha) \leq n$,
the fact that $\alpha + \theta(\alpha) = 0 \mod n$ implies
that $\alpha + \theta(\alpha) = n$.
In particular, $\theta(\alpha) \in \Sigma^+$.
A nearly identical argument shows that, if $\alpha \in \Sigma^-$,
then $\theta(\alpha) \in \Sigma^-$
and
$\alpha + \theta(\alpha) = - n$.
Therefore,
\begin{equation}\label{eqsum}
\sum_{\alpha \in \Sigma^+} \alpha = \frac{n}{2} \, \left|\Sigma^+\right|
\quad \mbox{and} \quad
\sum_{\alpha \in \Sigma^-} \alpha = - \frac{n}{2}\, \left|\Sigma^-\right|,
\end{equation}
where $\left|\Sigma^\pm \right|$ denotes the number of elements
(counted with multiplicity) in $\Sigma^\pm$.
Finally, similar arguments apply to $\theta'(\alpha)$ 
as long as $\theta'(\alpha) \neq n$, or equivalently
(since $(\theta')^2 = \id_{\Sigma}$) as long as $\alpha \neq \theta'(n)$.
In particular,
$\theta' (\alpha) \in \Sigma^+$ and $\alpha + \theta'(\alpha) = k$
for all $\alpha \in \Sigma^+ \smallsetminus \theta'(n)$, while
$\theta' (\alpha) \in \Sigma^-$ and $\alpha + \theta'(\alpha) = -k$
for all $\alpha \in \Sigma^- \smallsetminus \theta'(n)$.
We now need to consider the two possible cases:
\\
\\
\noindent
{\bf Case I:} \emph{Assume $\op{N}(k) \neq 0$}.

In this case, by the paragraph above
$k + \theta(k) = n$.
Since $n + \theta'(n) = 0 \mod k$, this implies that
$\theta(k) + \theta'(n) = 0 \mod k$.
On the other hand, by assumption, $-\theta(k) \not\in \Sigma$; in particular,
$\theta'(n) \neq - \theta(k)$.
Since $0 < \theta(k) \leq k$ and $-k < \theta'(n) \leq n$
these two facts imply that $\theta'(n) > 0$.
Therefore, by the preceding paragraph, $\theta'(\alpha) \in \Sigma^-$
and
$\alpha + \theta'(\alpha) =  - k$ for all $\alpha \in \Sigma^-$.
Hence, 
$$\sum_{\alpha \in \Sigma^-} \alpha 
= - \frac{k}{2} \, \left| \Sigma^-\right|.$$
Since $k \neq n$, this is a contradiction
unless $\Sigma^- =  \emptyset.$
But this contradicts the assumption that $\sum_{\alpha \in \Sigma} \alpha = 0$.
\\
\\
\noindent
{\bf Case II:} \emph{Assume $\op{N}(- k) \neq  0$}.

In this case, 
$-k + \theta(-k) = -n$,
and so  $\theta(-k) = - n \mod k$.
Therefore, we can choose $\theta'$ so that
$\theta'(n) = \theta(-k)$.
In particular, $\theta'(n) < 0$.
By an argument analogous  to that in the previous case, this
is impossible unless
$\Sigma^+ = \emptyset$.
Since $\alpha + \theta(\alpha) = -n$ for all $\alpha \in \Sigma^-$
and $\sum_{\alpha \in \Sigma} \alpha = 0$, this implies
that there  exist
natural numbers $a$ and $b$ so
that $\Sigma = \{a+b,-a,-b\}$.
\end{proof}

We are now ready to show that Theorem \ref{twopoints2} holds.

\begin{proof}[Proof of Theorem \ref{twopoints2}]
If $\dim M = 2$ then the quotient $M/S^1$ is a $1$-dimensional
manifold with boundary.
Hence, $M/S^1$ is homeomorphic to the
interval $[0,\,1] \subset \R$ and the
fixed points correspond to $\{0\}$ and $\{1\}$.
By Mayer-Vietoris, this implies that $H^1(M;\R) = 0$,
that is, $M$ is the $2$-sphere.
Note that in this case,
the ABBV localization formula implies that
there exists a natural number $a$ so that the weights at
the two fixed points are $a$ and $-a$.

Hence, we may assume that $\dim M > 2$  
and that the proposition is true for all
manifolds of dimension less than $\dim M$; we will prove
that it holds for $M$.

By quotienting out by the subgroup which acts trivially,
we may assume that the action is effective.
Let $p$ and $q$ be the fixed points.
Let $\Sigma_p$ and $\Sigma_q$ denote the set of weights
(counted with multiplicity) in the isotropy representations
$\op{T}_p M$ and $\op{T}_q M$, respectively.
Let
$\op{N}_p(\ell)$ and $\op{N}_q(\ell)$ denote the multiplicity of $\ell$
in $\Sigma_p$ and $\Sigma_q$, respectively, for all $\ell \in \Z$.
Finally, let  $\Lambda_p$ and $\Lambda_q$ be the product of the weights in $\Sigma_p$ and $\Sigma_q$,  respectively.

Let $n$ be the largest integer such that $\op{N}_p(n) + \op{N}_q(n) \neq 0$.
By relabeling if necessary, we may assume that $\op{N}_p(n) \neq 0$.
\\
\\
\noindent
{\bf Claim I:}  \emph{$\op{N}_p(\ell) \op{N}_p(-\ell) = \op{N}_q(\ell) 
\op{N}_q(-\ell) = 0$  for all $\ell \in \Z
\smallsetminus \{-1,0,1\}$}. 

Since $|\ell| > 1$,
$\dim  M^{\Z/(\ell)}  < \dim M$.  Moreover, 
the isotropy submanifold $M^{\Z/(\ell)}$ has at most two fixed points.
Therefore, the inductive hypothesis and Corollary~\ref{notone} together imply
that $\ell$ and $-\ell$ can't both be weights in 
$\op{T}_p M^{\Z/(\ell)}$
or in 
$\op{T}_q M^{\Z/(\ell)}$;  this proves the claim.
\\

\noindent
{\bf Claim II:} $\Sigma_p = - \Sigma_q$. 

By Claim I, $\op{N}_p(\ell) 
\op{N}_p(- \ell)  = \op{N}_q(\ell) \op{N}_q(-\ell)  = 0$ for all $\ell \in \Z \smallsetminus \{-1,0,1\}$.
By Lemma~\ref{last:lem}, 
$$
\op{N}_p(\ell) + \op{N}_q(\ell) = \op{N}_p(-\ell) + \op{N}_q(-\ell)
\quad \forall \ \ell \in \Z \smallsetminus \{-1,0,1\}; 
$$
therefore  $\op{N}_p(\ell) = 
\op{N}_q(-\ell)$ for all such $\ell$.
The claim then follows immediately from  Lemma~\ref{1}, 
which implies that $$\sum_{\ell=1}^\infty \op{N}_p(-\ell) = \lambda_p = \textstyle \frac{1}{2}\displaystyle
\dim M - \lambda_q = \sum_{\ell=1}^\infty \op{N}_q(\ell).$$
Here, $\lambda_p$ and $\lambda_q$ are the number of negative weights in $\Sigma_p$ and $\Sigma_q$, respectively.
\\

\noindent
{\bf Claim III:} \emph{$\Sigma_p = - \Sigma_p \mod \alpha$ for all
$\alpha \in \Sigma_p$}.

Consider a nonzero integer $\ell$ and
a component $Z \subset M^{\Z/(\ell)}$.
By Corollary~\ref{notone}, if $Z$ is not a single point then it either contains
both $p$ and $q$, or neither.
In the former case, $\Sigma_p = \Sigma_q \mod \ell$; see \cite[Lemma 2.6]{T}.
Hence, $\Sigma_p =  \Sigma_q \mod \alpha$ for every $\alpha \in \Sigma_p$. Therefore, 
the claim follows immediately from Claim II.
\\

\noindent
{\bf Claim IV:} 
\emph{ 
$\sum_{\alpha \in \Sigma_p} \alpha = 0$}.  

On the one hand,  Lemma~\ref{Azero} implies that since $\dim M >2$,
\begin{equation}
\label{eqletwo}
\op{c}_1(M)(p) = \op{c}_1(M)(q)
\quad \mbox{and} \quad \frac{1}{\Lambda_p} +
\frac{1}{\Lambda_q} = 0.
\end{equation}
Moreover, by Theorem~\ref{2}, $\op{c}_1(M)(p) + \op{c}_1(M)(q) = 0$;
therefore $$\sum_{\alpha \in \Sigma_p} \alpha = \op{c}_1(M)(p) =  0.$$
\\
\noindent
{\bf Claim V:} \emph{$n > 1$, $\op{N}_p(n) = 1$, and $\op{N}_p(\ell) = 0$ for all 
$\ell \in \Z$ such that $|\ell| >  n$}.

If the weights at $p$ and $q$ are all $1$ or $-1$, then the fact that
$\op{c}_1(M)(p) = \op{c}_1(M)(q)$ implies that $\Lambda_p = \Lambda_q$.
Since this contradicts \eqref{eqletwo}, there exists $\ell \neq \pm 1$ so
that $\op{N}_p(\ell) + \op{N}_q(\ell) \neq 0$. 
By Claim II and the definition of $n$, this implies that $n > 1$ and
$\op{N}_p(\ell) = 0$ for all $\ell \in \Z$ such that $|\ell| > n$.

Since $n > 1$,
$\dim  M^{\Z/(n)}  < \dim M $.  Moreover,   
the isotropy submanifold $M^{\Z/(n)}$ has  at most
two fixed points. Therefore, the inductive hypothesis and Corollary~\ref{notone}
together imply that $\op{N}_p(n) = 1$.
\\

\noindent 
{\bf Claim VI:} \emph{ $\Sigma_p = \{a+b, -a, -b\} \cup \bigcup_{i=1}^m \{1,-1\}$ for some natural numbers $a$, $b$, and $m$.}

Let $\Sigma'$ be the subset of $\Sigma_p$ formed by deleting the pair $\{1,-1\}$ 
as many times as possible,
that is,
$$\Sigma' =  \Sigma_p \smallsetminus  
\cup_{i=1}^{m} \{1,-1\}, \qquad \mbox{where} \  \ m = \min \{\op{N}_p(1),\op{N}_p(-1) \} .$$
By
Claims I, III, IV, and V,  the set $\Sigma'$ fulfills the requirements of Lemma~\ref{multi}; therefore,
the claim follows immediately from that result.
\\

\noindent
{\bf Claim VII:} $\Sigma_p = \{a+b,-a,-b\}$ and $\Sigma_q = \{a,b,-a-b\}$ for some natural numbers $a$ and $b$.
\\

By Claims II and  VI, 
$$\Sigma_p = - \Sigma_q = \{a+b,-a,-b\} \cup \bigcup_{i=1}^m \{1,-1\}$$
for some natural numbers $a$, $b$, and $m$.
Therefore, since $\dim(M) = 6 + 2m$,
to prove the claim is enough to we show that $\dim(M) \leq 6$.
Moreover, by a straightforward computation, 
\begin{equation*}
\op{c}_3(M)|_p = -  \op{c}_3(M)|_q =  (a+b) a b \, t^3
\quad \mbox{and} \quad
\Lambda_p  = - \Lambda_q = (-1)^{m} (a+b) a b .
\end{equation*}
Therefore, by the ABBV localization formula (Theorem \ref{ABBV}),
$$\int_M  \op{c}_3(M) = \frac{(-1)^m}{t^{2m}}  + \frac{(-1)^m}{t^{2m}}=\frac{2(-1)^m}{t^{2m}} \neq 0.  $$
As we explained in the beginning of Section~\ref{sec2}, this is impossible
unless $\dim M \leq 6$, as required.

\end{proof}

\bigskip\noindent
Alvaro Pelayo\\
University of California\---Berkeley \\
Mathematics Department,
970 Evans Hall $\#$ 3840 \\
Berkeley, CA 94720-3840, USA.\\
{\em E\--mail}: {apelayo@math.berkeley.edu

\bigskip\noindent
Susan Tolman\\
UI Urbana\--Champaign\\
Department of Mathematics\\
1409 W Green St, Urbana (USA)\\
e\--mail: {stolman@math.uiuc.edu}

\end{document}